\documentclass[10.5pt]{article}

\usepackage{graphicx, amssymb, latexsym, amsfonts, amsmath, lscape, amscd,
amsthm, color, epsfig, mathrsfs, tikz, enumerate, soul}

\newtheorem{theorem}{Theorem}[section]
\newtheorem{conjecture}[theorem]{Conjecture}
\newtheorem{corollary}[theorem]{Corollary}

\newtheorem{lemma}[theorem]{Lemma}
\newtheorem{proposition}[theorem]{Proposition}

\newcommand\DELETE[1]{}

\begin{document}


\title{{\bf On clique numbers of colored mixed graphs}}
\author{
{\sc Dipayan Chakraborty}$\,^{a}$, {\sc Sandip Das}$\,^{b}$, {\sc Soumen Nandi}$\,^{c}$, \\
{\sc Debdeep Roy}$\,^{d}$, {\sc Sagnik Sen}$\,^{a}$ \\
\mbox{}\\
{\small $(a)$ Indian Institute of Technology Dharwad, India}\\
{\small $(b)$ Indian Statistical Institute, Kolkata, India}\\
{\small $(c)$ Institute of Engineering and Management, India}\\
{\small $(d)$ ESSEC Business School Cergy, France}
}

\date{\today}

\maketitle

\begin{abstract}
An $(m,n)$-colored mixed graph, or simply, an $(m,n)$-graph is a graph having $m$ different types of arcs and $n$ different types of edges. A homomorphism of an $(m,n)$-graph $G$ to another $(m,n)$-graph $H$ is a vertex mapping that preserves adjacency; and the type and direction of the adjacency. 
An $(m,n)$-relative clique of $G$ is a vertex subset $R$ whose images are always distinct under any homomorphism of $G$ to any $H$. The maximum cardinality of an $(m,n)$-relative clique of a graph is called the $(m,n)$-relative clique number of the graph. In this article, we explore 
the $(m,n)$-relative clique numbers for three different families of graphs, namely, graphs having bounded maximum degree $\Delta$, subcubic graphs, partial $2$-trees and planar graphs and provide tight or close bounds in most cases. 
\end{abstract}


\section{Introduction}
In 2000, Ne\v{s}et\v{r}il and Raspaud~\cite{raspaud_and_nesetril} generalized 
the concepts of graph homomorphism, coloring and chromatic number of a graph by introducing colored homomorphisms and $(m,n)$-chromatic numbers of $(m,n)$-graphs. 
Later, Bensmail, Duffy, and Sen~\cite{bensmail2017analogues} introduced and studied two more parameters closely related to that of $(m,n)$-chromatic number, namely, the $(m,n)$-relative clique number and the $(m,n)$-absolute clique number. While the latter~\cite{bensmail2017analogues} had focused majorly on studying the $(m,n)$-absolute clique number, the focus of this article is to study the $(m,n)$-relative clique numbers for different families of graphs.

An \textit{$(m,n)$-colored mixed graph}, or simply, an $(m,n)$-graph $G$ is a graph having $m$ different types of arcs and $n$ different types of edges. The set of vertices,  arcs, edges, and the underlying undirected graph of $G$ are denoted by $V(G)$, $A(G)$, $E(G)$, and $und(G)$ respectively. In this article, we will restrict ourselves to $(m,n)$-graphs whose underlying graph is a simple graph only.

Let us fix a convenient convention: whenever we speak of an $(m,n)$-graph $G$, 
we imagine that the arcs of $G$ are labeled (colored) by one of the symbols $1, 2, \cdots, m$, and that the edges of $G$ are labeled by one of $m+1, m+2, \cdots, m+n$. 
Equivalently,  we shall think of $\sigma$ to be a function from  
$A(G) \to \{1,2,\cdots, m\}$ and $E(G) \to \{m+1, m+2, \cdots, m+n\}$ that denotes the labels on the arcs/edges. Thus, for an arc/edge $uv$ with label $k$, we have $\sigma(uv) = k$.  
Furthermore, if $uv$ is an arc, we shall set the convention of having 
$\sigma(vu) = -\sigma(uv)$, and thus increasing the domain and range of $\sigma$. However, as $uv$ is an edge if and only if $vu$ is an edge, we shall simply have $\sigma(vu) = \sigma(uv)$.

A (colored) \textit{homomorphism}~\cite{raspaud_and_nesetril} of an $(m,n)$-graph $G$ to an $(m,n)$-graph $H$ is a 
function $f: V(G) \to V(H)$ such that for any $uv \in A(G) \cup E(G)$, we have $f(u)f(v) \in A(H) \cup E(H)$ and $\sigma(uv) = \sigma(f(u)f(v))$.  Moreover, whenever a homomorphism $G$ to $H$ is understood to exist, we simply denote it by $G \to H$. 
Then, the \textit{$(m,n)$-chromatic number} of $G$ is given by
$$\chi_{m,n}(G) = \min\{|V(H)| : G \to H\}.$$

An  $(m,n)$-\textit{relative clique}~\cite{bensmail2017analogues} $R$ of $G$ is a vertex subset, i.e. $R \subseteq V(G)$, such that $f(u) \neq f(v)$ for all distinct $u,v \in R$ under any homomorphism 
$f$ of $G$ to any $(m,n)$-graph $H$. The \textit{$(m,n)$-relative clique number} of $G$ is then given by 
$$ \omega_{r(m,n)}(G) = \max\{|R|: R \text{ is an $(m,n)$-relative clique of $G$}\}.$$

On the other hand, 
an  \textit{$(m,n)$-absolute clique}~\cite{bensmail2017analogues} $A$ is an $(m,n)$-graph such that 
$\chi_{m,n}(A) = |V(A)|$; and the \textit{$(m,n)$-absolute clique number} of $G$ is given by $$\omega_{a(m,n)}(G) = \max\{|V(A)|: A \text{ is subgraph of $G$ and is an $(m,n)$-absolute clique}\}.$$

Given $p \in \{\chi_{m,n}, \omega_{r(m,n)}, \omega_{a(m,n)}\}$, for a family $\mathcal{F}$ of undirected simple graphs, the above mentioned parameters are defined as 
$$p(\mathcal{F}) = \max\{p(G): und(G) \in \mathcal{F}\}.$$

A direct observation~\cite{bensmail2017analogues} from the definitions gives us a relation among the above three parameters, namely, for any $(m,n)$-graph $G$,
$$\omega_{a(m,n)}(G) \leq \omega_{r(m,n)}(G) \leq \chi_{m,n}(G).$$

To date,  one of the major directions of research for studying colored homomorphisms of colored mixed graphs has been via establishing lower and upper bounds of $\chi_{m,n}(\mathcal{F})$ for different well-known graph families~\cite{raspaud_and_nesetril}\cite{Montejano2009363}\cite{humer}.
Apart from that, for specific (small) values of $(m,n)$, the bounds of $\omega_{r(m,n)}(\mathcal{F})$ and 
$\omega_{a(m,n)}(\mathcal{F})$ have also been studied~\cite{36}\cite{unique_oclique}\cite{das2018study}. However, for general values of $(m,n)$, the bounds for $\omega_{r(m,n)}(\mathcal{F})$ and $\omega_{a(m,n)}(\mathcal{F})$ have not been studied extensively except for in~\cite{bensmail2017analogues}. 
This work is a sequel to~\cite{bensmail2017analogues} where we majorly explore the parameter 
$\omega_{r(m,n)}(\mathcal{F})$ for different interesting families of graphs.

\medskip

In this article, we present bounds of the two parameters namely the $(m,n)$-absolute and the $(m,n)$-relative clique numbers of some graph classes $\mathcal{G}$. In Section~\ref{Prelims}, some definitions and notation are given. Also we prove some basic lemmas which play important roles to prove the theorems in the subsequent sections. In Section~\ref{Bounded Max Deg}, we prove upper bounds of the 
$(m,n)$-absolute and the $(m,n)$-relative clique numbers
of graphs having bounded maximum degree. 
In Section~\ref{Subcubic graph}, we study the same parameters for subcubic graphs and provide tight bounds for each $(m,n)$. In Section~\ref{Partial 2-tree}, we study the parameters  for partial $2$-trees and  also partial $2$-trees having girth at least $g$ and provide tight bounds for all $g \geq 3$. 
In Section~\ref{Planar graphs}, we provide bounds of the same parameters for planar graphs having girth at least $g$ for each value of $g \geq 3$. For $g = 3,4$ we provide loose upper and lower bounds, but we do prove tight bounds for all other values of $g$. 
In Section~\ref{Conclusions}, we provide conclusive remarks along with two interesting conjectures.

\section{Preliminaries}\label{Prelims}
We follow West~\cite{D.B.West} for all standard notation and terminology in graph theory, and as far as the non-standard notation are concerned, we introduce them here. If $uv$ is an arc/edge or if $vu$ is an arc having 
$\sigma(uv) = \alpha$, then $v$ is called an \textit{$\alpha$-neighbor} of $u$ and the set of all $\alpha$-neighbors of $u$ is denoted by $N^{\alpha}(u)$. Moreover, all $\alpha$-neighbors of $u$ are said to \textit{agree} on $u$ by the adjacency type $\alpha$.

A \textit{special $2$-path} connecting two vertices $u$ and $v$ is a $2$-path $uwv$ such that $\sigma(uw) \neq \sigma(vw)$. 
In an $(m,n)$-graph, a vertex $u$ is said to \emph{see} a vertex $v$ if they are either adjacent, or are connected by a special 2-path. If $u$ and $v$ are connected by a special 2-path with $w$ as the internal vertex, then it is said that $u~sees~v~through~w$ or equivalently, $v~sees~u~through~w$.

A particularly handy characterization of a relative clique is the following:

\begin{lemma}\cite{bensmail2017analogues}\label{lem 2path char}
Two distinct vertices of an $(m,n)$-graph $G$ are part of a relative clique if and only if they are either adjacent or connected by a special $2$-path in $G$. 
\end{lemma}

For the convenience of reference, we are going to name a  particular notion that will be used time and again in our proofs. Let $\mathcal{F}$ be a graph family. A \textit{critical $(m,n)$-relative clique} $H$ of $\mathcal{F}$ is an $(m,n)$-graph $H$ satisfying the following properties:
\begin{enumerate}[(i)]
\item   $und(H) \in \mathcal{F}$,

\item $\omega_{r(m,n)}(H) = \omega_{r(m,n)}(\mathcal{F})$,

\item $\omega_{r(m,n)}(H^*) < \omega_{r(m,n)}(\mathcal{F})$ if
$(|V(H^*)|, |E(und(H^{*}))|) < (|V(H)|, |E(und(H))|)$  in the dictionary ordering, where $H^* \in \mathcal{F}$.
\end{enumerate}

Whenever we consider a critical $(m,n)$-relative clique $H$ of graph family 
$\mathcal{F}$, we shall assume that $H$ contains a relative clique $R$ of cardinality
$\omega_{r(m,n)}(H)$. To that end, the vertices of $R$ are called \textit{good} vertices and those of $S = V(H) \setminus R$ are called the \textit{helper} vertices. Therefore, whenever we consider a critical relative clique $H$ for a graph family $\mathcal{F}$, the symbols $R$ and $S$ and the terms good and helper vertices should make sense. Due to the minimality of $H$, we can directly observe the following:

\begin{lemma}\label{obs S is indep}
Any critical $(m,n)$-relative clique $H$ of a subgraph closed family $\mathcal{F}$ is connected and $S$ is an independent set. 
\end{lemma}

\begin{proof}
If $H$ is not connected, then one of its components, say, $H_1$  
must have $\omega_{r(m,n)}(H) = \omega_{r(m,n)}(H_1)$. This contradicts the criticality of $H$, and thus $H$ is connected. 

If there is an arc or an edge $e$ between two vertices of $S$, then note that $e$ does not contribute to a pair of vertices of $R$ seeing one another. Thus, we will have 
$\omega_{r(m,n)}(H) = \omega_{r(m,n)}(H - \{e\})$, contradicting the criticality of $H$. Hence, $S$ is an independent set. 
\end{proof}

Let $G \in \mathcal{F}$ be any graph having a vertex $u$ of degree at most $k$. 
Let $G^*$ be the graph obtained by 
adding edges between every non-adjacent pair of neighbors of $u$ and then 
deleting the vertex $u$. If the family $\mathcal{F}$ is such that, for any $G \in \mathcal{F}$, we also have $G^* \in \mathcal{F}$, then we call $\mathcal{F}$ a 
 \textit{$k$-closed}
graph family. 

\begin{lemma}\label{obs good degree}
Let $H$ be a  critical $(m,n)$-relative clique of a $k$-closed family $\mathcal{F}$. 
Then any vertex in $H$ having degree $k$ or less is a good vertex. 
\end{lemma}

\begin{proof}
If there is a helper $h$ of degree $k$ or less in $H$,  then we can delete $h$
 and make its non-adjacent neighbors   adjacent to each other by adding some extra arcs/edges to obtain $H^* \in \mathcal{F}$. Note that, $R$ is still an 
$(m,n)$-relative clique in $H^*$. This contradicts the fact that $H$ is a  critical $(m,n)$-relative clique. 
\end{proof}

\section{Graphs with bounded maximum degree}\label{Bounded Max Deg}
Let $\mathcal{G}_{\Delta}$ denote the family of graphs 
having maximum degree $\Delta$. It is known that~\cite{DBLP:conf/caldam/DasNS17} the $(m,n)$-chromatic number of 
$\mathcal{G}_{\Delta}$ is lower bounded by $(2m+n)^{\frac{\Delta}{2}}$ and for connected graphs, is also upper bounded by $2(\Delta-1)^{2m+n}(2m+n)^{\Delta -1} +2$. This means that 
the $(m,n)$-chromatic number of $\mathcal{G}_{\Delta}$ is exponential in $\Delta$.

However, while studying the values of the $(m,n)$-absolute and the $(m,n)$-relative clique numbers 
of $\mathcal{G}_{\Delta}$, we find that their values are of a drastically smaller order of $\Omega(\Delta^{2})$. 

More interestingly, we find that the values of the parameters are also connected to the famous degree-diameter problem. So, if $\nu(2,\Delta)$ denotes the maximum number of vertices of a graph of diameter $2$ and maximum degree $\Delta$, the connection is the following.

\begin{theorem}\label{th rel delta}
For all $(m,n) \neq (0,1)$, we have
\begin{enumerate}[(i)]
\item $\nu(2,\Delta) = \omega_{a(m,n)}({\mathcal{G}_\Delta}) \leq \omega_{r(m,n)}({\mathcal{G}_\Delta}) \leq \Delta^2 +1$ for all $\Delta < 2m+n$.

\item $\omega_{a(m,n)}({\mathcal{G}_\Delta}) \leq 
\omega_{r(m,n)}({\mathcal{G}_\Delta}) \leq \Delta^2 +1$  
for $\Delta = 2m+n$.

\item $\omega_{a(m,n)}({\mathcal{G}_\Delta}) \leq 
\omega_{r(m,n)}({\mathcal{G}_\Delta}) \leq \lfloor\frac{2m+n-1}{2m+n}\Delta^2\rfloor+\Delta+1$  
for all $\Delta > 2m+n$.
\end{enumerate}
 \end{theorem}

The proof of the above result will be contained across several lemmas in this section. 

\begin{lemma}\label{lem delta2}
For all $(m,n) \neq (0,1)$ and all $\Delta \geq 0$, 
we have $\omega_{r(m,n)}({\mathcal{G}_\Delta}) \leq \Delta^2 +1$.
\end{lemma}

\begin{proof}
Let $G \in \mathcal{G}_\Delta$ with an $(m,n)$-relative clique $R$. It is enough to show that $|R| \leq \Delta^2 +1$.

Take any $u \in R$. As $\Delta$ is the maximum degree of $G$, $|N(u)| \leq \Delta$. 
Moreover, each vertex of $R \setminus N[u]$ must be adjacent to a vertex of $N(u)$ in order to see $u$. Observe that, any vertex of $N(u)$ can have at most  $(\Delta-1)$ neighbors other than $u$. Therefore, 
$$|R| \leq 1+|N(u)| + |\cup_{v \in N(u)} (N(v) \setminus \{u\})| \leq (\Delta +1) + \Delta(\Delta-1) = \Delta^2+1.$$ 
That completes the proof. 
\end{proof}

Given an $(m,n)$-graph $G$, let $G^2$ be the simple graph with $V(G^2)=V(G)$ and 
$E(G^2)= \{uv: \text{$u$ sees $v$ in $G$}\}$. 
Moreover, for any graph $G$ and any $X \subseteq V(G)$, we write $G[X]$ to denote the subgraph of $G$ induced by $X$. Then, we have,

 \begin{lemma}\label{lem discharging}
If $G$ is an $(m,n)$-colored mixed graph with maximum degree $ \Delta$, then 
$G^2$ is $ \Big( \lfloor{\frac{(2m+n-1)\Delta^2}{2m+n}}\rfloor+\Delta \Big)$-degenerate for all $(m,n) \neq (0,1)$.
 \end{lemma}

 \begin{proof}
 For convenience, let us assume that $p=2m+n$. Thus, it is enough to prove that, 
 for any $X \subseteq V(G)=V(G^2)$, the subgraph $G^2[X]$ has at least one vertex of degree less than or equal to  $(\lfloor{\frac{(p-1)\Delta^2}{p}}\rfloor+\Delta)$. We prove this using discharging.

 Assume that the initial charge $ ch(x)$ of a vertex $x$ of $G^2$ is
$$
ch(x)=
\begin{cases}
\frac{\Delta^2}{p} &\text{ if }  x \in X,\\
0  &\text{ if }  x \not\in X.
\end{cases}
$$
Thus, the total charge of the graph is 
$\sum_{x \in V(G^2)} ch(x) = \frac{\Delta^2|X|}{p}$.

Next we are going to present the discharging rule.

\medskip

 \textbf{(R1):}  Every vertex of $X$ donates a charge of $ \frac{\Delta}{p}$ to each of its neighbors in $G^2$.

\medskip

Let $ ch^*(x)$ be the updated charge of the vertices of $G^2$ after applying (R1). 
Observe that a vertex $x$ with $k$ neighbors in $X$ has a charge 
$ ch^*(x) \geq \frac{k \Delta}{p}$ as it has received $\frac{\Delta}{p}$ charge from each of its $k$ neighbors belonging to $X$ by (R1).

Let $\pi(x)$ be the number of special $2$-paths going through a degree $k$ vertex $x$ and linking two vertices in $X$. Let $\beta, \gamma \in \{-m,-(m-1), \cdots, -1, 1, 2, \cdots, m+n\}$ be distinct. Then, the number of special $2$-paths through $x$ due to its $\beta$-neighbors and $\gamma$-neighbors is
equal to $|N^{\beta}(x)|\cdot |N^{\gamma}(x)|$. Therefore, the total number of special $2$-paths through 
$x$ is equal to  
$$\pi(x) \leq  \sum  |N^{\beta}(x)|\cdot |N^{\gamma}(x)|,$$
where $\beta < \gamma$ are two of the $p$ symbols  
from $\{-m,-(m-1), \cdots, -1, 1, 2, \cdots, m+n\}$. 

Observe that, by maximizing the sum, we can conclude that 
$$\pi(x) \leq \bigg\lfloor {p \choose 2} \frac{k^2}{p^2} \bigg\rfloor = \bigg \lfloor \frac{(p-1)k^2}{2p} \bigg \rfloor . $$

Since $k \leq \Delta$ and $ch^*(x) \geq \frac{k \Delta}{p}$, we have
$$ \pi(x)\leq \bigg \lfloor \frac{(p-1)k^2}{2p}\bigg \rfloor \leq \frac{(p-1)k\Delta}{2p}\leq \frac{(p-1)ch^*(x)}{2}.$$

Note that, the above equation provides an upper bound on the number of  edges  in $G^2[X]$ 
that are there
due to special $2$-paths through $x$. On the other hand, by the Handshaking lemma, $G[X]$ has at most 
$\frac{\Delta}{2}|X|$ edges, those which $G^2[X]$ inherits from $G[X]$. Furthermore, as the total charge of the graph is constant, we have 
$\sum_{x \in V(G^2)} ch(x) =  \sum_{x \in V(G^2)} ch^*(x)$. Hence, the total number of edges in $G^2[X]$ will be at most 
$$\frac{\Delta}{2}|X| + \frac{p-1}{2}\sum_{x \in V(G^2)}ch^*(x) = 
\left(\frac{p-1}{p}\cdot\Delta^2+\Delta \right)\frac{|X|}{2}.$$

Thus, there exists at least one vertex in $G^2[X]$ having degree at most 
$$\left(\frac{p-1}{p}\cdot\Delta^2+\Delta \right).$$

This completes the proof. 
 \end{proof}

The above result is a generalization of a result due to Gon{\c{c}}alves, Raspaud and Shalu~\cite{shalu}. Next we are going to show that, indeed, we have 
$\nu(2,\Delta) = \omega_{a(m,n)}({\mathcal{G}_\Delta})$ for all $\Delta < 2m+n$.

\begin{lemma}\label{lem deg-diam}
For all $\Delta < 2m+n$, we have $\nu(2,\Delta) = \omega_{a(m,n)}({\mathcal{G}_\Delta})$.
\end{lemma}

\begin{proof}
Notice that there must exist an $(m,n)$-absolute clique 
$H \in \mathcal{G}_{\Delta}$ satisfying $|V(H)| = \omega_{a(m,n)}(\mathcal{G}_{\Delta})$. Therefore, $und(H)$ is a graph with diameter at most $2$ and maximum degree at most $\Delta$. Hence, 
$\omega_{a(m,n)}(\mathcal{G}_{\Delta}) = |V(H)| \leq \nu(2,\Delta)$.

On the other hand, there exists an undirected graph $U$ on $\nu(2,\Delta)$ vertices having diameter $2$ and maximum degree $\Delta$. We will construct an $(m,n)$-absolute clique $U^*$ whose underlying graph is $U$.

Observe that, $U$ is $(\Delta+1)$-edge colorable due to Vizing's Theorem~\cite{D.B.West}. Consider a proper edge coloring of $U$ using $1,2,\cdots, \Delta+1$. Notice that, as $2m+n > \Delta$, it is possible to find 
non-negative integers $m_1 \leq m$ and $n_1 \leq m+n-m_1$ satisfying $2m_1+n_1 = \Delta+1$.

Now, for each $i \leq m_1$, consider the subgraph of $U$ induced by the edge colors $(2i-1)$ and $2i$. For convenience, call it $U_i$. Notice that, each vertex of $U_i$ has degree at most $2$. This means that $U_i$ is a disjoint union of paths and cycles. Now, replace each of these paths/cycles by a directed path/cycle having arcs of label $i$ to (partially) construct $U^*$. 
Moreover, for $m_1 < j \leq m_1+n_1$, consider the subgraph $U_j$ induced by the edge color $2m_1+j$. 
The graph $U_j$ is a matching. We replace each edge of $U_j$ by an arc/edge of label $j$ (orientation of arcs can be chosen randomly) to construct $U^*$. 

Observe that, the $U^*$, so obtained, is an $(m,n)$-graph. 
Moreover, every vertex $x$ of $U^*$ has $\sigma(xy) \neq \sigma(xz)$ for distinct $y, z \in N(x)$. Therefore, $U^*$ is indeed an $(m,n)$-absolute clique having $und(U^*) = U$. 
 \end{proof}

Now, we are ready to prove Theorem~\ref{th rel delta}.

 \noindent \textit{Proof of Theorem~\ref{th rel delta}.}
The result follows directly from  Lemmas~\ref{lem delta2}, \ref{lem discharging}, and~\ref{lem deg-diam}. \qed
 \bigskip

 Das, Prabhu and Sen~\cite{das2018study} had shown that 
 $\omega_{a(1,0)}({\mathcal{G}_\Delta}) = \Omega(\Delta^2)$. 
 They proved the lower bound using examples. The same examples will imply a quadratic (in $\Delta$)
 lower bound for $\omega_{a(m,n)}({\mathcal{G}_\Delta})$ for all $(m,n)$ where $m \geq 1$. 
 Therefore, together with Theorem~\ref{th rel delta}, we can conclude that 
 $\omega_{a(m,n)}({\mathcal{G}_\Delta}) = \Theta(\Delta^2)$
 and $\omega_{r(m,n)}({\mathcal{G}_\Delta}) = \Theta(\Delta^2)$ whenever $m \geq 1$.

\section{Subcubic graphs}\label{Subcubic graph}
 In this section, we study the same  parameters for subcubic graphs and we manage to provide tight bounds in each case. That is, this section is dedicated to a special case $\Delta =3$ of the previous section.

   \begin{theorem}\label{th rel subcubic}
 For all $(m,n) \neq (0,1)$ we have 
 \begin{enumerate}[(i)]
\item $\omega_{a(1,0)}({\mathcal{G}_3}) = \omega_{r(1,0)}({\mathcal{G}_3}) = 7$~\cite{das2018study}.

\item $\omega_{a(0,n)}({\mathcal{G}_3}) = \omega_{r(0,n)}({\mathcal{G}_3}) = 8$ for $n = 2,3$.

\item $\omega_{a(m,n)}({\mathcal{G}_3}) = \omega_{r(m,n)}({\mathcal{G}_3}) = 10$ for $(m,n) = (1,1)$ and for $2m+n \geq 4$. 
\end{enumerate}
 \end{theorem} 
 
 We prove Theorem \ref{th rel subcubic} in a series of lemmas.

\begin{lemma}\label{lem big mn}
For $(m,n) = (1,1)$ and for $2m+n \geq 4$, we have 
$\nu(2,3) = \omega_{a(m,n)}({\mathcal{G}_3}) = \omega_{r(m,n)}({\mathcal{G}_3}) = 10$. 
\end{lemma}

\begin{proof}
We have 
$\omega_{r(m,n)}({\mathcal{G}_3}) \leq 10$ by Lemma~\ref{lem delta2}
for any 
$(m,n) \neq (0,1)$. 
Also, as 
$\nu(2,3) = 10$, uniquely realized by the Petersen Graph, we have 
$\omega_{a(m,n)}({\mathcal{G}_3}) =\omega_{r(m,n)}({\mathcal{G}_3}) = 10$
for $2m+n \geq 4$ due to Theorem~\ref{th rel delta}(i).  

In the case that $(m,n) = (1,1)$, the proof rests in showing that 
it is possible to convert the Petersen graph into an $(1,1)$-absolute clique (see Fig~\ref{fig Wagner(0,2)}($a$)).
\end{proof}

As $\omega_{a(1,0)}({\mathcal{G}_3}) =\omega_{r(1,0)}({\mathcal{G}_3}) = 7$
is proved by Das, Prabhu and Sen~\cite{das2018study}, we are only left with the cases when
$(m,n) = (0,2)$ and $(0,3)$.

\begin{figure}

\centering
\begin{tikzpicture}

\filldraw [black] (0,0) circle (2pt) {node[below]{}};
\filldraw [black] (2,0) circle (2pt) {node[below]{}};
\filldraw [black] (3,2) circle (2pt) {node[below]{}};
\filldraw [black] (1,4) circle (2pt) {node[left]{}};
\filldraw [black] (-1,2) circle (2pt) {node[above]{}};

\filldraw [black] (.5,.5) circle (2pt) {node[below]{}};
\filldraw [black] (1.5,.5) circle (2pt) {node[below]{}};
\filldraw [black] (0,2) circle (2pt) {node[below]{}};
\filldraw [black] (2,2) circle (2pt) {node[left]{}};
\filldraw [black] (1,3) circle (2pt) {node[above]{}};


\draw[-] (0,0) -- (.5,.5);

\draw[-] (2,0) -- (1.5,.5);

\draw[-] (3,2) -- (2,2);

\draw[-] (1,4) -- (1,3);

\draw[-] (-1,2) -- (0,2);


\draw[->] (.5,.5) -- (.75,1.75);
\draw[-] (.75,1.75) -- (1,3);

\draw[->] (1,3) -- (1.25,1.75);
\draw[-] (1.25,1.75) -- (1.5,.5);

\draw[->] (1.5,.5) -- (0.75,1.25);
\draw[-] (0.75,1.25) -- (0,2);

\draw[->] (0,2) -- (1,2);
\draw[-] (1,2) -- (2,2);

\draw[->] (2,2) -- (1.75,1.75);
\draw[-] (1.75,1.75) -- (.5,.5);


\draw[->] (0,0) -- (1,0);
\draw[-] (1,0) -- (2,0);

\draw[->] (2,0) -- (2.5,1);
\draw[-] (2.5,1) -- (3,2);

\draw[->] (3,2) -- (2,3);
\draw[-] (2,3) -- (1,4);

\draw[->] (1,4) -- (0,3);
\draw[-] (0,3) -- (-1,2);

\draw[->] (-1,2) -- (-0.5,1);
\draw[-] (-0.5,1) -- (0,0);


\node at (1,-0.5)   {$(a)$};


\filldraw [black] (6,0) circle (2pt) {node[below]{}};
\filldraw [black] (8,0) circle (2pt) {node[below]{}};
\filldraw [black] (9,1) circle (2pt) {node[below]{}};
\filldraw [black] (9,3) circle (2pt) {node[left]{}};
\filldraw [black] (8,4) circle (2pt) {node[above]{}};
\filldraw [black] (6,4) circle (2pt) {node[below]{}};
\filldraw [black] (5,3) circle (2pt) {node[below]{}};
\filldraw [black] (5,1) circle (2pt) {node[below]{}};

\draw[-] (6,0) -- (8,0);
\draw[-] (9,1) -- (9,3);
\draw[-] (8,4) -- (6,4);
\draw[-] (5,3) -- (5,1);

\draw[dashed] (8,0) -- (9,1);
\draw[dashed] (9,3) -- (8,4);
\draw[dashed] (6,4) -- (5,3);
\draw[dashed] (5,1) -- (6,0);

\draw[-] (6,0) -- (8,4);
\draw[dashed] (8,0) -- (6,4);

\draw[-] (9,1) -- (5,3);
\draw[dashed] (9,3) -- (5,1);


\node at (7,-0.5)   {$(b)$};

\end{tikzpicture}

\caption{(a) A $(1,1)$-absolute clique on $10$ vertices. (b) A $(0,2)$-absolute clique on $8$ vertices.}\label{fig Wagner(0,2)}
\end{figure}
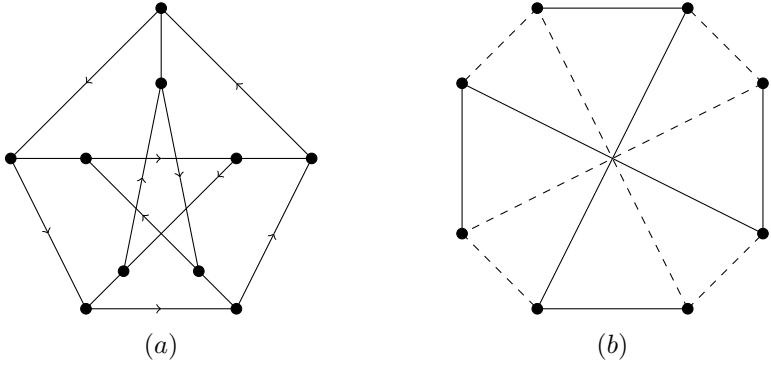

\begin{lemma}\label{lem 23 is 8}
We have $ \omega_{a(0,2)}(\mathcal{G}_3) \geq 8$ and 
$ \omega_{r(0,3)}(\mathcal{G}_3) \leq 8$. 
\end{lemma}

\begin{proof}
The first bound is implied by the $(0,2)$-absolute clique on $8$ vertices depicted in Fig.~\ref{fig Wagner(0,2)}($b$). 

\medskip

For the second bound, let $H$ be a critical $(0,3)$-relative clique of 
$\mathcal{G}_3$. By Lemma~\ref{lem delta2}, we have that $|R| \leq 10$. 

Let $u$ be a good vertex of $H$. Note that, it can have at most three good neighbors and six good second neighbors. Thus, to have $|R| = 10$, all nine of its neighbors and second neighbors of $u$ must be good. In particular, this implies that any good vertex 
is non-adjacent to a vertex of $S$ and hence, $S = \emptyset$. Therefore, $H$ must be a $(0,3)$-absolute clique. 

As the Petersen graph is the only cubic graph on $10$ vertices having diameter at most $2$, and as any two non-adjacent vertices of the Petersen graph are connected by a unique $2$-path, we can conclude that 
$H$ is a $(0,3)$-absolute clique if and only if there exists a proper $3$-edge-coloring of $H$. However, we know that the chromatic index of the Petersen graph is $4$ and so, it is not possible to have a $3$-edge-coloring of $H$. 

Therefore, $\omega_{r(0,3)}({\mathcal{G}_3}) \leq 9$.  Moreover, due to the handshaking lemma, there does not exist any cubic graph on $9$ vertices. Hence, 
$\omega_{r(0,3)}({\mathcal{G}_3}) \leq 8$. 
\end{proof}

Finally. we are ready to prove Theorem~\ref{th rel subcubic}. 

\medskip

 \noindent \textit{Proof of Theorem~\ref{th rel subcubic}.}
The proof of (ii) and (iii) follow directly from Lemmas~\ref{lem 23 is 8} and~\ref{lem big mn} respectively.
\qed

\section{Partial $2$-trees}\label{Partial 2-tree}
For the family of outerplanar graphs, Bensmail, Duffy and Sen~\cite{bensmail2017analogues} provided a tight bound of $3(2m+n)+1$ for its $(m,n)$-relative clique number. Here we consider a superfamily of it, namely, the family of partial $2$-trees, or equivalently, the family of $K_4$-minor-free graphs and provide tight bounds in this case too. Furthermore, we also provide tight bounds for some sparse subfamilies of partial $2$-trees. 
As, usually, many bounds and values of such parameters are the same for the families of 
outerplanar graphs and partial $2$-trees, the difference in our case comes as a surprise. 

Let $\mathcal{T}^2_g$ denote the family of partial $2$-trees having girth at least $g$, where \textit{girth} of a graph is the length of its smallest cycle. All bounds are tight in this section.

  \begin{theorem}\label{th 2tree}
 For all $(m,n) \neq (0,1)$, we have
 \begin{enumerate}[(i)]
\item $\omega_{a(m,n)}({\mathcal{T}^2_3}) = \omega_{r(m,n)}({\mathcal{T}^2_3}) = (2m+n)^2 + (2m+n) +1$.

\item $\omega_{a(m,n)}({\mathcal{T}^2_4}) = \omega_{r(m,n)}({\mathcal{T}^2_4}) = 
(2m+n)^2 + 2$.

\item $\omega_{a(m,n)}({\mathcal{T}^2_5}) = \omega_{r(m,n)}({\mathcal{T}^2_5}) = \max(2m+n+1,5)$ for $(m,n) \neq (0,2)$. 

\item $\omega_{a(0,2)}({\mathcal{T}^2_5}) = 3$ and $\omega_{r(0,2)}({\mathcal{T}^2_5}) = 4$.

\item $\omega_{a(m,n)}({\mathcal{T}^2_g}) = \omega_{r(m,n)}({\mathcal{T}^2_g}) = (2m+n)+1$ for all $g \geq 6$.
\end{enumerate}
 \end{theorem} 
 
We will handle the higher girth cases first. To begin with, let us prove a result for trees.

\begin{proposition}\label{prop mn trees}
For the family $\mathcal{T}$ of trees, we have 
$$\omega_{a(m,n)}(\mathcal{T}) = \omega_{r(m,n)}(\mathcal{T}) = (2m+n)+1$$ for all $(m,n) \neq (0,1)$. 
\end{proposition}

\begin{proof}
Let $H$ be a critical $(m,n)$-relative clique of $\mathcal{T}$. 
We know that helpers cannot have degree $1$ or less due to 
Lemma~\ref{obs good degree}.
Thus, there exists a good vertex $u$ of degree one, with a neighbor $v$ (say) in $H$. Every good second neighbor $w$ of $u$ must have 
$\sigma(uv) \neq \sigma(wv)$. Thus, two good neighbors of $v$ cannot have the same type of adjacency with $v$. Moreover, there cannot be any arc/edge 
or special $2$-path, other than the one through $v$, between the neighbors of $v$, as $H$ is a tree. Therefore, $v$ can have at most $(2m+n)$ good neighbors.

On the other hand, the star graph with $2m+n$ leaves, each having a distinct type of adjacency with the central vertex gives the lower bound. Hence we are done. 
\end{proof}

Next, let us handle the case of graphs having girth at least $7$. To do so, let us define $\mathcal{F}_g$ as the family of all graphs having  
 girth $g$.

\begin{proposition}\label{prop girth7}
 For $g \geq 7$, we have 
 $$\omega_{a(m,n)}(\mathcal{F}_{g-1}) = \omega_{r(m,n)}(\mathcal{F}_g) = (2m+n)+1$$
  for all $(m,n) \neq (0,1)$. 
\end{proposition}

\begin{proof}
Let $H$ be a critical $(m,n)$-relative clique of $\mathcal{F}_g$ having girth $g \geq 7$.
 Suppose $H$ contains a cycle $C = v_1v_2 \cdots v_kv_1$ of length $k$, where $k \geq 7$. By Lemma \ref{obs S is indep}, as two helpers can never be adjacent, $C$ must have a good vertex, say, $v_1$ without loss of any generality. 

Note that, if we add an arc/edge or a special $2$-path between $v_1$ and either $v_4$ or $v_5$, then a cycle of length $6$ or less is created. Thus, both $v_4$ and $v_5$ must be helpers. However, as $v_4$ and $v_5$ are adjacent, this is a contradiction to Lemma~\ref{obs S is indep}. Thus, $H$ does not contain a cycle, which implies that $H$ is a tree.

If however, $H \in \mathcal{F}_{g-1}$ is an $(m,n)$-absolute clique and $C =  v_1v_2 \cdots v_lv_1$, where $l \geq 6$, is a cycle, we can see that it is not possible for $v_1$ to see $v_4$ without violating the girth restrictions. Thus, $H$ is a tree in this case too.

Hence, the proof follows from Proposition~\ref{prop mn trees}.
\end{proof}

Now, let us examine the situation for $(m,n)$-relative clique number of girth $6$ graphs. 

\begin{lemma}\label{lem girth6 relative situation}
Let $H$ be a critical $(m,n)$-relative clique of $\mathcal{F}_6$ having a $6$-cycle $C = v_1v_2 \cdots v_6v_1$, where $v_1$ is a good vertex. Then, $v_3$ and $v_5$ are also good vertices while $v_2, v_4$ and $v_6$ are helpers. Moreover, any good vertex $v$, other than $v_1$, $v_3$ and $v_5$, must be connected to each of the latter by internally disjoint special $2$-paths.
\end{lemma}

\begin{proof}
Notice that two antipodal vertices of $C$ cannot be good vertices, as an arc/edge or special $2$-path connecting them will contradict the girth restrictions. Thus, $v_4$ is a helper as it is antipodal to $v_1$. This implies that $v_3$ and $v_5$ are good due to Lemma~\ref{obs S is indep}. By the same argument then, $v_2$ and $v_6$  are helpers, as they are antipodals to $v_5$ and $v_3$ respectively,.

As for the last part of the Lemma, observe that the only way for a vertex $v$ that is not part of the cycle $C$ to be able to see $v_1, v_3$ and $v_5$ are through distinct helpers due to the girth restrictions. 
\end{proof}

\begin{lemma}\label{lem rel girth6 2tree}
 We have 
 $$\omega_{r(m,n)}(\mathcal{T}^2_6) = (2m+n)+1$$
  for all $(m,n) \neq (0,1)$. 
\end{lemma}

\begin{proof}
Let $H$ be a critical $(m,n)$-relative clique of $\mathcal{T}^2_6$. If $H$ contains a $6$-cycle and 
$\omega_{r(m,n)}(H) \geq 4$, then it will force a $K_4$-minor in 
$H$ by Lemma~\ref{lem girth6 relative situation}. 
As partial $2$-trees are $K_4$-minor-free, this is a contradiction. 
Therefore, if $H$ contains a $6$-cycle, we must have $\omega_{r(m,n)}(H) \leq 3 \leq (2m+n)+1$ for all $(m,n) \neq (0,1)$.

On the other hand, if $H$ does not contain any $6$-cycle, 
then $H \in \mathcal{T}^2_g$, where $g \geq 7$, and we are done due to Proposition~\ref{prop girth7}.
\end{proof}

This, along with Proposition \ref{prop mn trees}, proves Theorem \ref{th 2tree}(v). Now, we are going to handle the case where the girth is at least $5$. 
However, during the proof,
we shall  show something useful 
for planar graphs also.

\begin{lemma}\label{lem girth5 structure}
Let $H$ be a critical $(m,n)$-relative clique of $\mathcal{F}_5$ having a $5$-cycle $C = v_1v_2 \cdots v_5v_1$ where $v_1, v_3$ and $v_4$ are good vertices. 
Then any other good vertex $v$, not contained in $C$,  
must be either adjacent or connected by 
 internally disjoint special  $2$-paths with $v_1, v_3$ and $v_4$.
 Moreover, the above-mentioned paths connecting $v$ to $v_3$ and $v_4$ must be special $2$-paths only. 
\end{lemma}

\begin{proof}
Follows from the girth restrictions directly. 
\end{proof}

Now, we are ready to handle the girth $5$ case. 

\begin{lemma}\label{lem rel girth5 2tree}
 We have 
 $$\omega_{a(m,n)}(\mathcal{T}^2_5) = \omega_{r(m,n)}(\mathcal{T}^2_5) = \max(2m+n+1, 5)$$
  for all $(m,n) \neq (0,1), (0,2)$. 
\end{lemma}

\begin{proof}

Observe that $\nu (2,2)=5$ and it is realized by the $5$-cycle. So, taking $U$ as a $5$-cycle in the proof of Lemma \ref{lem deg-diam}, we can transform the $5$-cycle into an $(m,n)$-absolute clique for $2m+n \geq 3$. Moreover, for $(m,n) = (1,0)$, the directed $5$-cycle
is a $(1,0)$-absolute clique. Thus, $5 \leq \omega_{a(m,n)}(\mathcal{T}_5^2) \leq \omega_{r(m,n)}(\mathcal{T}_5^2)$ for $(m,n) \neq (0,1), (0,2)$.

Moreover, if $H$ contains a $5$-cycle, then the former must also contain at least $3$ good vertices due to Lemma~\ref{obs S is indep}. Thus, if $\omega_{r(m,n)}(H) \geq 6$, then $H$ will have a $K_4$-minor by Lemma~\ref{lem girth5 structure}. This is a contradiction, as $H$ is a partial $2$-tree and is, hence, $K_4$-minor-free. Therefore, we are done in the case when $H$ contains a $5$-cycle.

However, if $H$ does not 
contain a $5$ cycle, then we are done by Lemma~\ref{lem rel girth6 2tree}, Proposition \ref{prop mn trees} and Proposition \ref{prop girth7}.  
\end{proof}

This proves Theorem \ref{th 2tree}(iii) and leaves us with the corner cases of figuring out the values of 
$\omega_{a(0,2)}(\mathcal{T}^2_5)$ and $\omega_{r(0,2)}(\mathcal{T}^2_5)$.

\begin{lemma}\label{lem 02-rel girth5 2tree}
 We have 
 $$\omega_{a(0,2)}(\mathcal{T}^2_5) = 3 \text{ and } \omega_{r(0,2)}(\mathcal{T}^2_5) = 4.$$
\end{lemma}

\begin{proof}
Let $H \in \mathcal{T}^2_5$ be a $(0,2)$-absolute clique. As $5$-cycle is not an underlying graph of a $(0,2)$-absolute clique, by Lemma~\ref{lem rel girth6 2tree}, Proposition~\ref{prop mn trees} and Proposition \ref{prop girth7}, $\omega_{a(0,2)}(\mathcal{T}^2_5) = 3$.

Next, let $C=v_1v_2v_3v_4v_5v_1$ be a $5$-cycle and let
$\sigma(v_1v_{2})= \sigma(v_3v_{4}) = \sigma(v_5v_{1}) = 1$ and 
$\sigma(v_2v_{3})= \sigma(v_4v_{5}) =2$.

Then, $\{v_1,v_2,v_3,v_4\}$ is a $(0,2)$-relative clique of $C$. Since $C \in \mathcal{T}^2_5$, $\omega_{r(m,n)}(\mathcal{T}^2_5) \geq 4$. Now, let $H$ be a critical $(0,2)$-relative clique of $\mathcal{T}^2_5$ such that it has a $5$-cycle $C$. Then, $C$ has at least $3$ good vertices by Lemma~\ref{obs S is indep} and at most $4$ good vertices by the girth restrictions. So, $|R| \geq 5$ implies that there exists a vertex $v \in R$ which is not part of $C$. This, in turn, implies a $K_4$-minor in $H$ by Lemma~\ref{lem girth5 structure}, a contradiction to $H$ being a partial $2$-tree. Thus, $\omega_{r(0,2)}(\mathcal{T}^2_5) = |R| \leq 4$ when $H$ has a $5$-cycle. If however, $H$ does not have a $5$-cycle, then, again, using Lemma~\ref{lem rel girth6 2tree} and Proposition \ref{prop girth7} $\omega_{r(0,2)}(\mathcal{T}^2_5)=|R| \leq 3 < 4$, a contradiction to $\omega_{r(m,n)}(\mathcal{T}^2_5) \geq 4$; and so, the case does not arise.
\end{proof}

This proves Theorem \ref{th 2tree}(iv).

\medskip


As we have finished dealing with the higher girth cases, let us move to the lower girth cases now. 
We start with proving the lower bounds.

\begin{lemma}\label{lem 2tree lower bound}
There exists an $(m,n)$-absolute clique on $(2m+n)^2+(2m+n)+1$ vertices whose underlying graph is a  partial $2$-tree. 
\end{lemma}

\begin{proof}
Take an $(m,n)$-graph $X$ whose underlying graph is $K_{1,2m+n}$ with the 
degree $(2m+n)$ vertex being $u$. Also, assign adjacencies in such a way that 
$\sigma(uv_1) \neq \sigma(uv_2)$ for all $v_1, v_2 \in N(u)$. This is an 
$(m,n)$-absolute clique on $(2m+n)+1$ vertices.

Now, take $(2m+n)$ copies of $X$ and name them $X_{-m}, X_{-(m-1)}, \cdots X_{-1}, X_1, X_2, \cdots , X_{m+n}$. Futhermore, take a new vertex $x$ and make it adjacent to all vertices of $X_i$ in such a way that $N^{i}(x) = X_{i}$ for all 
$i \in \{-m, -(m-1), \cdots, -1, 1, 2, \cdots, m+n\}$. Note that this new graph so obtained is an $(m,n)$-absolute clique on $(2m+n)^2+(2m+n)+1$ vertices whose underlying graph is a  partial $2$-tree. 
\end{proof}

Next is the turn to prove the lower bound for triangle-free partial $2$-trees. 

\begin{lemma}\label{lem 2tree triangle free lower bound}
There exists an $(m,n)$-absolute clique on $(2m+n)^2+2$ vertices whose underlying graph is a triangle-free partial $2$-tree. 
\end{lemma}

\begin{proof}
Take an $(m,n)$-graph $X$ whose underlying graph is $K_{2,(2m+n)^2}$ with the 
degree $(2m+n)^2$ vertices being $u_1$ and $u_2$. Now, assign adjacencies in such a way that 
$$(\sigma(u_1v_1),\sigma(u_2v_1)) \neq (\sigma(u_1v_2), \sigma(u_2v_2))$$
for all $v_1, v_2 \in N(u_1) \cap N(u_2)$. This is an 
$(m,n)$-absolute clique on $(2m+n)^2+2$ vertices whose underlying graph is a triangle-free partial $2$-tree. 
\end{proof}

With both the lower bounds proved, we will now engage in proving the upper bounds. 
However, before that, we will describe some common groundwork for both the cases.


 Let $H_g$ be a critical $(m,n)$-relative clique of $\mathcal{T}^2_g$, where 
 $g \in \{3,4\}$. Therefore, by Lemma~\ref{obs good degree}, $H_g$ contains a degree two good vertex $u$ with neighbors $u_1$ and $u_2$ (say). 
 Moreover, 
 every good vertex is adjacent to $u_1$ or $u_2$ by Lemma~\ref{lem 2path char}. 
 Let 
 $Y$ be the set of good vertices adjacent to both $u_1$ and $u_2$, 
 $Y_1$ be the set of good vertices adjacent to $u_1$ but not to $u_2$, and 
 $Y_2$ be the set of good vertices adjacent to $u_2$ but not to $u_1$. 
 Note that, as $u \in Y$, we have $|Y| \geq 1$. 
 Also, without loss of generality, assume that $|Y_1| \geq |Y_2|$. 
Based on these notations and nomenclatures, we present some lemmas that will lead us to the coveted proofs for the upper bounds.

\begin{lemma}\label{lem at most p+1}
It is possible to have at most $(2m+n+1)$ good neighbors in $N^{\alpha}(v)$ for some $v \in V(H_g)$,
where $\alpha \in \{-m, -(m-1), \cdots, -1, 1, 2, \cdots, m+n\}$.  Moreover, if $N^{\alpha}(v)$ has exactly $(2m+n+1)$ good neighbors, then the latter induce a star. 
\end{lemma}

\begin{proof}
Suppose that a particular planar embedding of $H_g$ is given and $|N^{\alpha}(v) \cap R| \geq 2m+n+2$.
Let the good $\alpha$-neighbors $v_1, v_2, \cdots, v_t$ ($t \geq 2m+n+2$) of $v$ be arranged in a clockwise manner around $v$ in the said embedding. Thus, in particular, $v_1$ sees $v_3$ either by some special $2$-path through some $w$ (which may be one of the $v_i$s) or directly by adjacency. Hence, $v_2$ is forced to see all $v_j$s ($j \geq 4$)
either through $v_1$ or $v_3$ or $w$. In any case, all $v_i$s are adjacent to some vertex $w$ (which may be one of the $v_i$s) other than $v$. Therefore, $N^{\alpha}(v) \cap N(w)$ contains at least $(2m+n+1)$ good vertices. This implies that at least two good vertices are inside $N^{\alpha}(v) \cap N^{\beta}(w)$ 
where $\beta \in \{-m, -(m-1), \cdots, -1, 1, 2, \cdots, m+n\}$. Then, these two vertices see each other either by being adjacent or by some special $2$-path through a $w' \not\in \{v,w\}$. Notice that this creates a $K_4$-minor in $H_g$. Thus, $N^{\alpha}(v)$ can have at most $(2m+n+1)$ good neighbors.

For the moreover part, notice that, if $w$ as above is not one of the $v_i$s and if $N^{\alpha}(v)$
contains only $(2m+n+1)$ neighbors, then also 
$N^{\alpha}(v) \cap N(w)$ is forced to contain at least $(2m+n+1)$ good vertices, eventually, resulting in 
a $K_4$-minor in $H_g$. 
\end{proof}

A direct corollary follows:

\begin{corollary}\label{cor at most p+1 triangle free}
It is possible to have at most $2m+n$ good neighbors in $N^{\alpha}(v)$ for some $v \in V(H_4)$ where $\alpha \in \{-m, -(m-1), \cdots, -1, 1, 2, \cdots, m+n\}$.   
\end{corollary}

\begin{proof}
Having $2m+n+1$ good neighbors in $N^{\alpha}(v)$ will result in triangles in $H_4$ due to the moreover part of Lemma~\ref{lem at most p+1}. 
\end{proof}

After the above two structural results, we are now ready to 
handle the particular situation when $Y_2 = 0$.

 \begin{lemma}\label{lem C2=0}
If $|Y_2| = 0$, then $H_3$ has at most $(2m+n)^2+(2m+n)+1$ good vertices. 
\end{lemma}

\begin{proof}
Let $H_3'$ be the graph obtained by 
adding an edge between $u_1$ and $u_2$ (if already not present) in $H_3$. 
As $u_1$ and $u_2$ are both adjacent to the vertex $u$ of degree two,  observe that, $H_3'$ is also 
$K_4$-minor-free like $H_3$. 
Now, since $|Y_2| = 0$, therefore $u_1$ is adjacent to all good vertices in $H_3'$. 
Thus, by Lemma~\ref{lem at most p+1}, $H_3$ can have at most $(2m+n)^2+(2m+n)+1$ good vertices. 
\end{proof}

\begin{lemma}\label{lem C2=0 tf}
If $|Y_2| = 0$, then $H_4$ has at most $(2m+n)^2+2$ good vertices. 
\end{lemma}

\begin{proof}
If $|Y_2| = 0$, then $u_1$ is adjacent to all good vertices, except maybe $u_2$. 
Thus, by Corollary~\ref{cor at most p+1 triangle free}, $N(u_1)$ contains at most $(2m+n)^2$ good vertices.  
\end{proof}

 Now, to handle the cases when $|C_2| \geq 1$, we need to establish a few more structural properties of $H_g$.

 \begin{lemma}\label{lem at most 1 common}
If $|N(v_1) \cap N(v_2)| \geq 3$ for some $v_1, v_2 \in V(H_g)$, then it is possible to have at most one good vertex in $N^{\alpha}(v_1) \cap N^{\beta}(v_2)$, where $\alpha, \beta \in \{-m, -(m-1), \cdots, -1, 1, 2, \cdots, m+n\}$.  
\end{lemma}

\begin{proof}
Suppose $|N(v_1) \cap N(v_2)| \geq 3$ and $N^{\alpha}(v_1) \cap N^{\beta}(v_2)$ contains at least two good vertices. Then, these good vertices must see each other either by adjacency or by some special $2$-path through a $w \not\in \{v_1, v_2\}$, thus forcing a $K_4$-minor in $H_g$. 
\end{proof}

 \begin{lemma}\label{lem alpha-alpha forbidden}
Let $|Y_2| \geq 1$. In $H_g$, if $u_i$ has a good $\alpha$-neighbor in $Y$, then $u_i$ cannot have
a good $\alpha$-neighbor in $Y_i$ for all $i \in \{1,2\}$.  
\end{lemma}

 \begin{proof}
 We shall prove this for $i = 1$, as the proof for $i = 2$ is similar.

 Let $u'$ be a good $\alpha$-neighbor of $u_1$ in $Y$ and let $v$ be a 
  good $\alpha$-neighbor of $u_1$ in $Y_1$. As the only way for $u$ to see $v$ is by a special $2$-path through
  $u_1$, it is not possible to have $u = u'$. 
  Thus, $u'$ sees $v$ either by adjacency or by some special $2$-path through a $w \not\in \{u,u_1,u_2\}$. 
  
  Moreover, as $|Y_2| \geq 1$, there is a good vertex $v'$ in $C_2$. 
  Notice that, it is not possible to have $v' = w$, as that forces a $K_4$-minor in $H_g$. 
  However, as $v'$ must see $v$, the options for that to happen are either via an arc/edge between them or  via a special $2$-path between them through some $w' \notin \{u,u_1,u_2\}$. Now, both the options produce $K_4$-minors in $H_g$, expect in the case that $w'=u'$ in $H_3$; and so, we turn our attention to it. 
   
   So, let $g=3$ and $w'=u'$. 
   Notice that, in this case, each vertex of $Y_1$ must see each vertex of $Y_2$ through $u'$ to avoid 
   creating a $K_4$-minor. Moreover, since $H_3$ is $K_4$-minor-free, this makes $Y, Y_1$ and $Y_2$ independent, and disallows any adjacency or special $2$-paths between vertices of
\begin{enumerate}[(i)]
\item $Y \setminus \{u'\}$ and $Y_i$ and
\item $Y_1$ and $Y_2$.
\end{enumerate}
   
   Thus, the only way for the vertices of $(Y \setminus \{u'\}) \cup Y_1 \cup Y_2$ to see 
   each other is by a special $2$-path through exactly one of $u', u_1$ and $u_2$. Thus, the adjacency types of the arcs/edges between $u'$ and vertices in $Y_1$ are different from those of the arcs/edges between $u'$ and vertices in $Y_2$. Similarly, 
   the adjacency types of the arcs/edges between $u_i$ and vertices in $(Y \setminus \{u'\})$ are different from those of the arcs/edges between $u_i$ and vertices in $Y_i$.

So, assume that $u_i$ has $p_i$ different types of adjacencies 
with the vertices in $(Y \setminus \{u'\})$; and suppose that $u'$ has $q$ different types of adjacencies with the vertices in $Y_1$. Therefore, taking $p=2m+n$ for convenience, we must have 
$$|R| \leq p_1p_2 + (p-p_1)q + (p-p_2)(p-q) + 3 \leq p^2 + p_1p_2 - p_1q - p_2(p-q) + 3.$$
Now, without loss of generality, assuming that $p_1 \geq p_2$, we have 
$$|R| \leq p^2 + p_1p_2 - p_2p + 3 = p^2 + p_2(p_1 - p) + 3 \leq p^2+2 < p^2+p+1,$$
as $1 \leq p_2 \leq p_1 < p$ since $u \in Y$. This is a contradiction, as $|R| \geq p^2+p+1$ due to Lemma~\ref{lem 2tree lower bound}. 
 \end{proof}

 These two structural results now enable us to handle the case $|Y_2|=1$.

 \begin{lemma}\label{lem C2=1}
If $|Y_2| = 1$, then $H_3$ has at most $(2m+n)^2+(2m+n)+1$ good vertices. 
\end{lemma}

\begin{proof}
Let $v$ be the good vertex in $Y_2$. This implies that $u_2$ can have at most $(2m+n-1)$ different
types of adjacencies with the good vertices in $Y$ due to Lemma~\ref{lem alpha-alpha forbidden}. 
Moreover, let $u_1$ have $p_1 \leq 2m+n$ different types of adjacencies with  the good vertices in $Y$. Thus, $u_1$ can have at most $(2m+n-p_1)$ different types of adjacencies with  the good vertices in $Y_1$.
Furthermore, notice that, as $u \in Y$, we have $p_1 \geq 1$. 

Thus, there can be at most $p_1(2m+n-1)$ good vertices in $Y$ due to Lemma~\ref{lem at most 1 common} and at most $(2m+n-p_1)(2m+n+1)$ good vertices in $Y_1$ due to Lemma~\ref{lem at most p+1}. 
The only other possible good vertices in $H_3$ are $u_1, u_2$, and $v$. Thus, taking $p=2m+n$ for convenience, the total number of good vertices in $H_3$ is
$$|R| \leq p_1(p-1) + (p-p_1)(p+1) + 3 = p^2 +p - 2p_1 + 3 \leq p^2 +p +1.$$
The last inequality uses the fact that $p_1 \geq 1$, as $u \in Y$. 
\end{proof}

 \begin{lemma}\label{lem C2=1 tf}
If $|C_2| = 1$, then $H_4$ has at most $(2m+n)^2+2$ good vertices. 
\end{lemma}

\begin{proof}
Let $v$ be the good vertex in $Y_2$. This implies that $u_2$ can have at most $(2m+n-1)$ different
types of adjacencies with the good vertices in $Y$ due to Lemma~\ref{lem alpha-alpha forbidden}. 
Moreover, let $u_1$ has $p_1 \leq 2m+n$ different types of adjacencies with  the good vertices in $Y$. Thus, $u_1$ can have at most $(2m+n-p_1)$ different types of adjacencies with  the good vertices in $Y_1$.
Furthermore, notice that, as $u \in Y$, we have $p_1 \geq 1$. 

Thus, there can be at most $p_1(2m+n-1)$ good vertices in $Y$ due to Lemma~\ref{lem at most 1 common} and 
there can be at most $(2m+n-p_1)(2m+n)$ good vertices in $Y_1$ due to Lemma~\ref{lem at most p+1}. 
The only other possible good vertices in $H_4$ are $u_1, u_2$, and $v$. Thus, taking $p=2m+n$ for convenience, the total number of good vertices in $H_4$ is
$$|R| \leq p_1(p-1) + (p-p_1)p + 3 = p^2 - p_1 + 3 \leq p^2 +2.$$
The last inequality uses the fact that $p_1 \geq 1$, as $u \in Y$. 
\end{proof}

We are now left with one last structural result which will allow us to 
handle the case when $|Y_2| \geq 2$.

 \begin{lemma}\label{lem C1-C2 universal}
If $|Y_2| \geq 2$, then in $H_g$, all good vertices of $Y_1$ and $Y_2$ are adjacent to a vertex $w$. 
Moreover, if $w$ has a good $\alpha$-neighbor in $Y_1$, then $w$ cannot have
a good $\alpha$-neighbor in $Y_2$.
\end{lemma}

 \begin{proof}
 Let us fix a planar embedding of $H_g$ and assume that in that embedding $u_1, u_2$ are arranged in a clockwise order around $u$. Furthermore, let $u_{11}, u_{12}, \cdots, u_{1t}$ denote the good neighbors of $u_1$
 in $Y_1$, arranged in an anti-clockwise order around $u_1$  and 
 let $u_{21}, u_{22}, \cdots, u_{2s}$ denote the good neighbors of $u_2$
 in $Y_2$ arranged in a clockwise order around $u_2$. 
 Notice that, $u_{11}$ must see $u_{22}$ either by adjacency or by a special $2$-path through a vertex $w$ (which maybe one of the $u_{ij}$s). This will force all 
 $u_{ij}$s to be adjacent to $w$ in order to see each other and also maintain the planarity of $H_g$.

For the moreover part, observe that, any adjacency or a special $2$-path between a $u_{1i}$ and a $u_{2j}$ through some vertex other than $w$ produces a $K_4$-minor in $H_g$. 
 \end{proof}

Finally, we are ready to prove the upper bound for the case when $|Y_2| \geq 2$.

 \begin{lemma}\label{lem C2>1}
If $|Y_2| \geq 2$, then $H_g$ has at most $(2m+n)^2+2$ good vertices. 
\end{lemma}

\begin{proof}
According to Lemma~\ref{lem C1-C2 universal}, there exists a $w$ which is adjacent to all good vertices in $Y_1 \cup Y_2$. Moreover, suppose that $w$ has $q_1 \leq 2m+n$ different types of adjacencies with the good vertices in $Y_1$. Then, $w$ can have at most $(2m+n-q_1)$ different types of adjacencies with the good vertices in $Y_2$ due to the moreover part of Lemma~\ref{lem C1-C2 universal}.

Further, suppose that $u_i$ has $p_i \leq 2m+n$ different types of adjacencies with the good vertices in $Y$, where $i \in \{1,2\}$. This implies that $u_i$ has at most $(2m+n-p_i)$ different types of adjacencies with the good vertices in $Y_i$
due to Lemma~\ref{lem alpha-alpha forbidden}. Now, let us calculate the possible number of good vertices in $H_g$ using Lemma~\ref{lem at most 1 common} and along with the observation that the only possible good vertices not contained in $Y \cup Y_1 \cup Y_2$ are $u_1$ and $u_2$. 
Therefore,  taking $p=2m+n$ for convenience, we have 
$$|R| \leq p_1p_2 + (p-p_1)q_1 + (p-p_2)(p-q_1) + 2 \leq p^2 + p_1p_2 - p_1q_1 -p_2(p-q_1) + 2.$$
Let us, with out loss of any generality, assume that $p_1 \geq p_2$ in the above inequality. Thus, we get
$$|R| \leq p^2 +p_1p_2 - p_2q_1 - p_2(p-q_1) +2 = p^2 + p_2(p_1-p)+2 \leq p^2+2.$$
The last inequality uses the fact that $1 \leq p_2 \leq p_1 < p$, as $u \in Y$.
\end{proof}

 Finally, we are ready to prove Theorem~\ref{th 2tree}.

\medskip 

\begin{proof}[Proof of Theorem \ref{th 2tree}]
The proof of (i) directly follows from  Lemmas~\ref{lem 2tree lower bound}, \ref{lem C2=0}, \ref{lem C2=1}, and~\ref{lem C2>1}. The proof of (ii) directly follows from  Lemmas~\ref{lem 2tree triangle free lower bound}, \ref{lem C2=0 tf}, \ref{lem C2=1 tf}, and~\ref{lem C2>1}. The proof of (iii) directly follows from  Lemma~\ref{lem rel girth5 2tree}. The proof of (iv) directly follows from  Lemma~\ref{lem 02-rel girth5 2tree} and the proof of (v) directly follows from  Proposition~\ref{prop mn trees} and Lemma~\ref{lem rel girth6 2tree}.
\end{proof}

\section{Planar graphs}\label{Planar graphs}
Let $\mathcal{P}_g$ denote the family of planar graphs having girth at least $g$. We provide lower and upper bounds for the $(m,n)$-relative clique numbers of $\mathcal{P}_g$ for each value of $g \geq 3$.

  \begin{theorem}\label{th planar}
 For all $(m,n) \neq (0,1)$ we have 
 \begin{enumerate}[(i)]
\item $3(2m+n)^2 + (2m+n) +1 \leq \omega_{a(m,n)}({\mathcal{P}_3}) \leq \omega_{r(m,n)}({\mathcal{P}_3}) \leq 42(2m+n)^2 -11$.

\item  $(2m+n)^2 +2 = \omega_{a(m,n)}({\mathcal{P}_4}) \leq \omega_{r(m,n)}({\mathcal{P}_4}) \leq 14(2m+n)^2 +1$.

\item $\max(2m+n+1, 5)= \omega_{a(m,n)}({\mathcal{P}_5}) \leq \omega_{r(m,n)}({\mathcal{P}_5}) = \max(2m+n+1, 6) $.

\item $2m+n+1= \omega_{a(m,n)}({\mathcal{P}_6}) \leq \omega_{r(m,n)}({\mathcal{P}_6}) = \max(2m+n+1, 4) $.

\item $\omega_{a(m,n)}({\mathcal{P}_g}) = \omega_{r(m,n)}({\mathcal{P}_g}) = (2m+n)+1 $ for $g \geq 7$.
\end{enumerate}
 \end{theorem} 
 
 Let us start with proving some structural properties first.

 \begin{lemma}\label{lem forbidden agree indep p}
 Let $H$ be a critical $(m,n)$-relative clique of $\mathcal{P}_g$
  for $g=3$ or $4$. Then the number of independent good vertices of $H$ 
  agreeing on another vertex $v$ can be at most $2(2m+n)$, where $2m+n \geq 3$.  
  \end{lemma}

\begin{proof}
Let $v$ be a vertex of $H$ and let $v_1, v_2, \cdots, v_t$ be independent good vertices agreeing on $v$ by the adjacency type $\alpha$.  Assume a fixed planar embedding of $H$ and suppose, without loss of generality, that in the said embedding 
$v_1, v_2, \cdots, v_t$ are arranged around $v$ in a clockwise manner. Also, note that $t \geq 7$ as $2m+n \geq 3$.

Now $v_1$ must see $v_3$ by a special $2$-path through some $w$ which is not equal to either $v$ or any of the $v_i$s. Notice that, every $v_i$, for $i \geq 4$, must see $v_2$ through $w$, as $H$ is planar. Therefore, given this structure, every $v_i$ must see $v_{i+k}$ by a special $2$-path through $w$ for all $k \geq 2$ (here, the $+$ opertion on the indices is taken modulo $t$). Therefore, $N^{\beta}(w)$ can contain at most two vertices from $v_1, v_2, \cdots, v_t$ for any $\beta \in \{-m, -(m-1), \cdots, -1, 1, 2, \cdots, (m+n)\}$ and hence, the result.
\end{proof}

As a corollary to the above lemma, we have the following:

\begin{lemma}\label{lem forbidden indep p2}
Let $H$ be a critical $(m,n)$-relative clique of $\mathcal{P}_g$
  for $g=3$ or $4$. Then the number of independent good neighbors of a vertex $v$ of $H$ is at most $2(2m+n)^2$. 
\end{lemma}

\begin{proof}
Follows from Lemma~\ref{lem forbidden agree indep p} and the pigeonhole principle.
\end{proof}

Now, we are ready to prove the upper bounds of Theorem~\ref{th planar}(i) and (ii). 

\begin{lemma}\label{lem rel girth3 planar}
For $(m,n)$ such that $2m+n \geq 3$, we have $\omega_{r(m,n)}({\mathcal{P}_3}) \leq 42(2m+n)^2 -11$.
\end{lemma}

\begin{proof}
Towards a contradiction, let $\omega_{r(m,n)}({\mathcal{P}_3}) > 42(2m+n)^2 -11$ and let $H$ be a critical 
$(m,n)$-relative clique of $\mathcal{P}_3$. We now claim that it is always possible to find a good vertex in $H$ whose degree in $und(H)$ is at most $7$. The proof of the claim goes as follows:

As planar graphs are $5$-degenerate, there exists a vertex in $H$ of degree at most $5$. If that vertex is a good vertex, then we are done. Otherwise, every good vertex of $H$ has a minimum degree of $6$. Moreover, as planar graphs are $3$-closed, every helper vertex of $H$ must have a minimum degree of $4$ due to Lemma~\ref{obs good degree}. Now, Jendrol' and Voss~\cite{jendrol2013light} had shown that any planar graph with minimum degree at least four must have an edge $uv$ such that $deg(u)+deg(v) \leq 11$. So, if every good vertex of $H$ is of degree at least $6$, then at least one of $u$ and $v$ must be a helper. So, without loss of generality, let $u$ be a helper vertex. This implies that $v$ must be a good vertex by Lemma~\ref{obs S is indep}. Moreover, as $u$ is of degree at least $4$, it implies that $v$ has a degree of at most $7$ in $und(H)$; and that proves the claim.

 Notice that the good vertices in the second neighborhood of $v$ induce an outerplanar subgraph of $und(H)$ and is, hence, $3$-colorable. As there are at least 
 $42(2m+n)^2 -10$ good vertices in $H$ and at most $7$ good vertices can be neighbors of $v$, there must be at least $14(2m+n)^2 -6$ independent good vertices in the second neighborhood of $v$ by the pigeonhole principle. Then, again by the pigeonhole principle, at least one neighbor of $v$ must have a minimum of $2(2m+n)^2+1$ independent good neighbors (including $v$), which is a contradiction to 
 Lemma~\ref{lem forbidden indep p2}. 
\end{proof}

A similar proof works for establishing the upper bound for the relative clique number for triangle-free planar graphs as well. 

\begin{lemma}\label{lem rel girth4 planar}
For $(m,n)$ such that $2m+n \geq 3$, we have $\omega_{r(m,n)}({\mathcal{P}_4}) \leq 14(2m+n)^2 +1$.
\end{lemma}

\begin{proof}
Let  $H$ be a critical 
$(m,n)$-relative clique of $\mathcal{P}_4$.  Now, delete each helper in $H$ of degree at most three and make its neighbors adjacent to each other (if not adjacent already). This will result in a planar graph which is not necessarily triangle-free. However, again using 
Jendrol' and Voss~\cite{jendrol2013light}, one can find a good vertex $v$ in the new graph having degree at most $7$. 
Observe that, the same vertex $v$ has at most degree $7$ in $H$ as well.

As $H$ is triangle-free,  any neighborhood of a vertex is independent in $H$. 
Thus, each neighbor of $v$ can have at most $2(2m+n)^2$ good neighbors
including $v$; and hence, the result. 
\end{proof}

Next, let us find the exact value for a triangle-free planar $(m,n)$-absolute clique. 

\begin{lemma}
For $(m,n)$ such that $2m+n \geq 3$, we have $\omega_{a(m,n)}({\mathcal{P}_4}) \leq (2m+n)^2 +2$.
\end{lemma}

\begin{proof}
Take a multigraph having three vertices $x,y,z$ with $a$ edges between $x$ and $y$,
$b$ edges between $y$ and $z$, and one edge between $x$ and $z$. Now subdivide all its edges, except $xz$, exactly once. The simple graph so obtained is denoted by $C_{a,b}$.

As per Plesn\'{i}k~\cite{plesnik}, the only triangle-free planar graphs having diameter two are:
$K_{1,t}, K_{2,t}$ and $C_{a,b}$. 
As a triangle-free planar $(m,n)$-absolute clique is one among these three types of graphs, observe that its maximum order is realized by  $K_{2,(2m+n)^2}$. 
\end{proof}

After this, let us consider the higher girth cases and start with planar graphs with girth at least $6$.

\begin{lemma}\label{lem rel girth6 planar}
For $2m+n \geq 3$, we have
 $\omega_{r(m,n)}(\mathcal{P}_6) = 2m+n+1.$
  \end{lemma}

\begin{proof}
For $2m+n \geq 3$,  $\omega_{r(m,n)}(\mathcal{P}_6) \geq (2m+n)+1$ due to 
  Lemma~\ref{lem rel girth6 2tree} as partial $2$-trees are planar graphs in particular.

Let $H$ be a critical $(m,n)$-relative clique of $\mathcal{P}_6$. If $H$ contains a $6$-cycle, then $\omega_{r(m,n)}(H) \geq 5$ will force a $K_5$-minor in 
$H$ by Lemma~\ref{lem girth6 relative situation}. As planar graphs cannot have a $K_5$-minor, $\omega_{r(m,n)}(H) \leq 4$, if $H$ contains a $6$-cycle.

On the other hand, if $H$ does not contain any $6$-cycle, 
then $H \in \mathcal{F}_7$ and we are done due to Proposition~\ref{prop girth7}
and the fact that $(2m+n)+1 \geq 4$.
\end{proof}

Next, we consider the family of planar graphs with girth at least $5$.

\begin{lemma}\label{lem rel girth5 planar}
 For $2m+n \geq 3$, we have
 $\omega_{r(m,n)}(\mathcal{P}_5) = \max(2m+n+1, 6).$
  \end{lemma}

\begin{proof}
For $2m+n \geq 3$,  $\omega_{r(m,n)}(\mathcal{P}_5) \geq (2m+n)+1$ due to 
  Lemma~\ref{lem rel girth6 2tree}, as partial $2$-trees are planar graphs in particular. On the other hand, $\omega_{r(m,n)}(\mathcal{P}_5) \geq 6$ due to the following construction:
  take an $(m,n)$-absolute clique whose underlying graph is the $5$-cycle (existance follows from Lemma~\ref{lem deg-diam}) and connect it to a sixth vertex using internally disjoint special $2$-paths.

Let $H$ be a critical $(m,n)$-relative clique of $\mathcal{P}_5$. If $H$ contains a $5$-cycle, then $\omega_{r(m,n)}(H) \geq 7$ will force a $K_5$-minor in 
$H$ by Lemma~\ref{lem girth5 structure}. As planar graphs cannot have a $K_5$-minor, $\omega_{r(m,n)}(H) \leq 6$, if $H$ contains a $5$-cycle.

On the other hand, if $H$ does not contain any $5$-cycle, 
then we are done due to Lemma~\ref{lem rel girth6 planar}. 
\end{proof}

Finally, we are ready to prove Theorem~\ref{th planar}

\medskip 

\begin{proof}[Proof of Theorem \ref{th planar}]
The proof of (i) directly follows from  Lemma~\ref{lem rel girth3 planar}. The proof of (ii) directly follows from  Lemma~\ref{lem rel girth4 planar}. The proof of (iii) directly follows from  Lemma~\ref{lem rel girth5 planar}. The proof of (iv) directly follows from  Lemma~\ref{lem rel girth6 planar} and the proof of (v) directly follows from  Proposition~\ref{prop girth7}.

\end{proof}

\section{Conclusions}\label{Conclusions}

This work may be regarded as the first systematic study of the $(m,n)$-relative clique number. We explored the $(m,n)$-relative clique number of subcubic graphs, graphs with bounded degree, partial $2$-trees and planar graphs of girth at least $g$, where $g \geq 3$. In case of subcubic graphs and partial $2$-trees having girth $g$, where $g \geq 3$, we have provided the exact bounds for all cases. In case of planar graphs having girth $g$, we were unable to provide tight bounds for the cases $g=3,4$. However, based on our experience of finding the bounds in those cases we would like to conjecture the following tight bounds.

\begin{conjecture}
For the family $\mathcal{P}_3$ of planar graphs,
$$\omega_{a(m,n)}({\mathcal{P}_3}) = \omega_{r(m,n)}({\mathcal{P}_3})=3(2m+n)^2 + (2m+n) +1.$$
\end{conjecture}

Notice that our conjecture strengthens the conjecture by Bensmail, Duffy and Sen~\cite{bensmail2017analogues} which claimed only 
$\omega_{a(m,n)}({\mathcal{P}_3}) = 3(2m+n)^2 + (2m+n) +1$. 
We make a similar conjecture for triangle-free planar graphs also. Notice that, as in this case we have already found out the exact value of $\omega_{a(m,n)}({\mathcal{P}_4})$, the conjecture only concerns the 
value of $\omega_{r(m,n)}({\mathcal{P}_4})$.

\begin{conjecture}
For the family $\mathcal{P}_4$ of triangle-free planar graphs,
$$\omega_{r(m,n)}({\mathcal{P}_4})=2(2m+n)^2 +2.$$
\end{conjecture}

\bibliographystyle{abbrv}
\bibliography{NSS14}

\end{document}